\documentclass{article}
\usepackage{graphicx} 
\usepackage{amssymb,amsthm,amsmath}
\usepackage{hyperref} 

\providecommand{\keywords}[1]
{\small	
  \textbf{Keywords:} #1}
\providecommand{\msc}[2]
{\small	
  \textbf{MSC:} #2}

\title{ \textbf{Anticipated BSDEs driven by fractional Brownian motion with time-delayed generator}}

\author{Pei Zhang$^{1,2}$, Nur Anisah Mohamed $^{1}$, Adriana Irawati Nur Ibrahim  $^{1}$}

\footnotetext[1]{ Institute of Mathematical Sciences, Faculty of Science, Universiti Malaya, 50603 Kuala Lumpur, Malaysia.}
\footnotetext[2]{ School of Mathematics and Statistics, Suzhou University, Suzhou 234000, China.} 
\begin{document}
\maketitle
\begin{abstract}
    This paper discusses a new type of anticipated backward stochastic differential equation with a time-delayed generator (DABSDEs, for short) driven by fractional Brownian motion, also known as fractional BSDEs, with Hurst parameter $H\in(1/2,1)$, which extends the results of the anticipated backward stochastic differential equation to the case of the drive is fractional Brownian motion instead of a standard Brownian motion  and in which the generator considers not only the present and future times but also the past time. By using the fixed point theorem, we will demonstrate the existence and uniqueness of the solutions to these equations. Moreover, we shall establish a comparison theorem for the solutions.
\end{abstract}
\\
\msc{}{60H10; 60H20; 60G22}\\

\section{Introduction}

Since Pardoux and Peng \cite{Par90} first proposed a general form of non-linear backward stochastic differential equations (BSDEs) in 1990, the theoretical research of BSDEs has developed rapidly. In our research, we are looking at the case where there exists a pair of adapted processes $(Y_{\cdot}, Z_{\cdot})$ that satisfy the following type of BSDE
\begin{eqnarray*}
Y_t=\xi+\int_{t}^{T}f(s,Y_s,Z_s)\,{\rm d}s-\int_{t}^{T}Z_s\,{\rm d}B_s, \ \ 0\leq t\leq T,
\end{eqnarray*}
where $\xi$ is the terminal value,  $f$ is the generator related to the present time, and $B_s$ is a standard Brownian process. After the already mentioned celebrated work of Pardoux and Peng, the interest in BSDEs has increased, mainly due to the connection of these tools with stochastic control and PDEs, a connection that will be stated clearly soon, for example, various BSDEs models and the uniqueness and existence of the solutions to these models (Bahlali et al. \cite{Ba02}; Abdelhadiet al. \cite{Abd22} Zhang et al. \cite{Zhang22}), the numerical solution of BSDEs (Ma et al. \cite{Ma02}; Gobet et al. \cite{Go05}; Zhao et al. \cite{Zhao14}), the relationship between BSDEs and partial differential equations (PDEs) (Ren and Xia \cite{Ren06}; Pardoux and R{\u{a}}{\c{s}}canu \cite{Par14}), and the numerous applications of BSDEs in various areas including optimal control, finance, biology, and physics (for examples, refer to \cite{Peng99,Li09,Kar97}).

With the further development of the BSDEs theory, an increasing number of models are being studied. Peng and Yang \cite{Peng09} initially discussed a fundamental class of BSDEs in 2009, namely, anticipated BSDEs, where  \begin{eqnarray*}
 \left\{
\begin{array}{ll}
    \displaystyle
    Y_{t}=\xi_{T}+\int_{t}^{T} f(s,Y_{s},Z_{s},Y_{s+\delta(s) },Z_{s+\zeta(s)})\,{\rm d}{s}- \int_{t}^{T}Z_{s}{\rm d}{B_s}, \quad 0\leq t\leq T;\\ 
    \displaystyle Y_{t}=\xi_t, \quad T\leq t\leq T+K;
\\
    \displaystyle Z_{t}=\eta_t, \quad T\leq t\leq T+K.
\end{array} \right.
\end{eqnarray*}
\newpage

The two deterministic $\mathbb{R}^{+}$-valued continuous functions $\delta(s), \zeta(s)$  defined on $[0, T]$ satisfy (i) $t\le t+\delta(t)\le T+K$, $t\le t+\zeta(t)\le T+K$, and (ii) $\int_{t}^{T}f(s+\delta(s))\,{\rm d}{s}\le L\int_{t}^{T+K}f(s){\rm d}{s}$,  $\int_{t}^{T}f(s+\zeta(s))\,{\rm d}{s}\le L\int_{t}^{T+K}f(s){\rm d}{s}$; the authors also demonstrated the existence and uniqueness of the solution to the above equations. Feng \cite{Feng16} investigated the uniqueness and existence of the solution of an anticipated BSDE with a reflecting boundary. Zhang et al.\cite{Zhang23} obtained some results of mean-filed anticipated BSDEs with a time-delayed generator. Wang and Cui \cite{Wang22} also proposed a new type of differential equation called anticipated backward doubly stochastic differential equation; the authors solved certain stochastic control problems by utilizing the duality between anticipated BSDEs and stochastic differential delay equations.

On the other hand, Delong and Imkeller \cite{Del10} addressed BSDEs with time-delayed generators as follows:
\begin{eqnarray*}
  Y_{t}=\xi+\int_{t}^{T}f(s,Y_{s-u(s)},Z_{s-v(s)})\,{\rm d}{s}- \int_{t}^{T}Z_{s}{\rm d}{B_s}, \quad 0\leq t\leq T,
\end{eqnarray*}

where $f$ is a generator that depends on the past value of a solution and $0\le u(s)\le T$, $0\le v(s)\le T$. As a generalization of Delong and Imkeller \cite{Del10} or Peng and Yang \cite{Peng09}, He et al. \cite{He20} investigated a type of delay and anticipated BSDEs.
Under partial information, Zhuang \cite{Zhuang17} studied non-zero and differential games for the anticipated forward-backward stochastic differential delay equation, which can be used to resolve a problem involving the management of time-delayed pension funds with non-linear expectations.

On the other hand,  first introduced by Kolmogorove\cite{Kol40} in 1940, the fractional Brownian motion (fBm, for short) $B_{t}^{H}$ with Hurst parameter $H\in(1/2,1)$ is a centered Gaussian process with good properties such as self-similarity and long-wall correlation, making it reasonable and efficient to use fBm as a random noise term in stochastic models in the fields of communication engineering, finance and economics.
It is therefore important to study the existence uniqueness and stability of solutions of BSDEs driven by fBm.
 
In 2009, Hu and Peng \cite{Hu09} first proposed BSDEs driven by fractional Brownian motion, that is, the fractional BSDE, has the general form of:
\begin{eqnarray*}
  Y_{t}=\xi+\int_{t}^{T}f(s,Y_{s},Z_{s})\,{\rm d}{s}- \int_{t}^{T}Z_{s}{\rm d}{B_{s}^{H}}, \quad 0\leq t\leq T.
\end{eqnarray*}
Henceforth, the amount of study being done on fractional BSDEs is progressively growing; for example, Borkowska \cite{Bor13} studied generalized BSDEs driven by fractional Brownian motion. Douissi et al. \cite{Dou19} showed a new kind of mean-field anticipated BSDE driven by fractional Brownian motion. Besides, Wen and Shi\cite{Wen17} focused on anticipated BSDEs driven by fractional Brownian motion, while Wen \cite{Wen22} discussed fractional BSDEs with delayed generator.

However, under the condition of BSDEs driven by fractional Brownian motion, the case where the generator considers not only the current time and the future time but also the past time has not been studied yet. Therefore, our study will focus on studying the BSDEs of this case to enrich the theory of BSDEs. This study might then encourage researchers to investigate stochastic optimal control problems more realistically.

Based on the motivations discussed above, an essential and meaningful question is, if we construct the anticipated BSDEs with a time-delayed generator driven by fBm with Hurst parameter $H\in(1/2,1)$, how can we prove the existence and uniqueness of its solution? In addition, what about the relative comparison theorem? In this work we are interested in the following anticipated backward stochastic differential equation with a time-delayed generator (DABSDEs for short) driven by fBm
\begin{equation}\label{eq1}
\left\{
\begin{array}{ll}
  \displaystyle -{\rm d}Y_{t}=f\left(t,Y_{t-d_{1}(t)},Z_{t-d_{2}(t)},Y_{t},Z_{t},Y_{t+d_{3}(t)},Z_{t+d_{4}(t)}\right)\,{\rm d}{t}
-Z_{t}{\rm d}{B_{t}^{H}}, \quad 0\leq t\leq T;
\\
 \displaystyle Y_{t}=\xi_t, \quad T\leq t\leq T+K;
\\
 \displaystyle Z_{t}=\eta_t, \quad T\leq t\leq T+K.
\end{array} \right.
  \end{equation}
  
 The rest of the framework for this study is organized as follows. Section 2 introduces some basic information on the new BSDEs model that we are proposing, which is the fractional DABSDEs. In Section 3, by using the fixed point theorem, we demonstrate the existence and uniqueness of the solutions for this type of BSDE. Then, the comparison theorem of the solutions for this kind of model is obtained in Section 4. 

\section{Preliminaries} 
\newtheorem{theorem}{Theorem}[section]
\newtheorem{definition}[theorem]{Definition}
\newtheorem{lemma}[theorem]{Lemma}
\newtheorem{proposition}[theorem]{Proposition}
Let us start with some definitions of the 
the problem at hand, fractional Brownian motion, assumptions, and some basic results of propositions that will be used throughout the paper. The readers are advised to study papers like Decreusefond and Üstünel \cite{Dec99}, Duncan et al. \cite{Dun00} and Hu \cite{Hu05}, etc. for a more in-depth discussion. 
\subsection{Preliminaries on the Fractional Brownian Motion}
Let $B^{H}=\{B_{t
}^{H}, t\ge 0\}$ be a fractional Brownian motion with Hurst parameter $H\in (0,1)$, which defined on a complete probability space$(\Omega,\mathcal{F},P)$ with filtration $\mathcal{F}$ generated by fBm $\{B_{t
}^{H}\}_{t\ge 0}$, its covariance kernel is given by
\begin{eqnarray*}
  R_{H}(s,t)= E_{H}\left[B_{t}^{H}B_{s
}^{H}\right]=\frac{1}{2} \left(t^{2H}+s^{2H}-\left |t-s  \right |^{2H} \right), \ \  s,t\ge 0.
\end{eqnarray*}

When $H=\frac{1}{2}$, it becomes a standard Brownian motion, when $0\le H \le\frac{1}{2}$, $B_{t}^{H}$ displays negative correlation property while it exhibits a positive correlation and long-range dependence properties when  $\frac{1}{2}\le H \le 1 $. Let $H\ge\frac{1}{2}$, and to simplify the presentation, we only discuss the one-dimensional case throughout this study. Next, consider the following definitions as Hu \cite{Hu05} shown in 2005.
At first, we define 
\begin{eqnarray*}
\left \langle \xi,\eta \right \rangle _{t}=\int_{0}^{t} \int_{0}^{t} \phi (u-v)\xi_u\eta_v{\rm d}u{\rm d}v,\ \ and~  \left \| \xi \right \|^{2}_{t}=\left \langle \xi,\xi\right \rangle _{t}, 
\end{eqnarray*}
where $\xi$ and $\eta$ are two continuous functions on $[0,T]$, for $x\in\mathbb{R},~\phi(x)=2H(2H-1)|x|^{2H-2}$. Then$\left \langle \xi,\eta \right \rangle _{t}$ is a Hilbert scalar product. Let $\Theta_t$ be the completion of the continuous functions under this Hilbert norm.

Let $\xi_1, \xi_2,\dots ,\xi_k,\dots $ be continuous functions on $[0,T]$, $f$ is a polynomial of $n$ variables. Donate $\mathcal{P}_T$ is the set of all polynomials of fBm over $[0,T]$ which contains all elements of the form
\begin{eqnarray*}
  F(\omega)=f\left( \int_{0}^{T}\xi_{1}(t){\rm d}B_{t}^{H},\dots, \int_{0}^{T}\xi_{n}(t){\rm d}B_{t}^{H}\right).
\end{eqnarray*}

Define the Malliavin derivative $D_{s}^{H}$ of a polynomial functional $F$ from $L^{2}(\Omega, \mathcal{F}, P) \to (\Omega, \mathcal{F}, \Theta_t)$ as following
\begin{eqnarray*} D_{s}^{H}F=\sum_{k=1}^{n} \frac{\partial f}{\partial x_k} \left( \int_{0}^{T}\xi_{1}(t){\rm d}B_{t}^{H},\dots, \int_{0}^{T}\xi_{n}(t){\rm d}B_{t}^{H}\right)\xi_{k}(s), \ \ 0\le s \le T.
\end{eqnarray*}

For $F\in \mathcal{P}_T$, let $\mathbb{D}_{1,2}^{H}$ be the completion of $\mathcal{P}_T$ with respect to the norm
\begin{eqnarray*}
    \left \| F \right \|_{H,1, 2}:= E\left[ \left(\left \| F \right \|_{T}^2\right)^\frac{1}{2}\right]+E\left[ \left(\left \| D_{s}^{H}F \right \|_{T}\right)^{\frac{1}{2}}\right].
\end{eqnarray*}

Introduce another derivative as well
\begin{equation*}   \mathbb{D}_{t}^{H}F=\int_{0}^{T}\phi(t-s)D_{s}^{H}F{\rm d}s.
\end{equation*}

\begin{proposition}\label{proposition1}
(Hu \cite{Hu05}, Proposition 6.25)
If $F_s: (\Omega,\mathcal{F},P)\to \Theta_{t}$ is a continuous process such that $E\left[\left \|F  \right \|_{T}^{2}+ \int_{0}^{T}\int_{0}^{T}|\mathbb{D}_{s}^{H}F_t|^{2}{\rm d}s{\rm d}t \right]\le \infty $, denoted as $F_s\in \mathbb{L}^{1,2}_{H}$, then the It\^{o} type stochastic integral $\int_{0}^{T}F_{s}{\rm d}B_{s}^{H}$ exists in $L^2(\Omega,\mathcal{F},P)$ and
\begin{equation*}
    E\left[\int_{0}^{T}F_{s}{\rm d}B_{s}^{H} \right]=0, 
 \end{equation*}   
\begin{equation*}
     E\left[\int_{0}^{T}F_{s}{\rm d}B_{s}^{H} \right]^{2}= E\left[ \left \| F \right \| _{T}^{2}+\int_{0}^{T}\int_{0}^{T}\mathbb{D}_{s}^{H}F_t\mathbb{D}_{t}^{H}F_s{\rm d}s{\rm d}t \right].
\end{equation*}
\end{proposition}

\begin{proposition}\label{proposition2}
(Hu \cite{Hu05}, Theorem 10.3)
   For $i=1,2$, let $g_{i}(s) \in \mathbb{D}_{1,2}^{H}$, and  $f_{i}(s),g_{i}(s), s\in[0,T],$ be real
valued stochastic processes satisfying $E\left[\int_{0}^{T}(|f_{i}(s)|^{2}+|g_{i}(s)|^{2}){\rm d}s \right]<\infty$. Suppose that $\mathbb{D}_{t}^{H}g_{i}(s)$ are continuously differentiable with respect to $(s,t)\in [0,T]\times[0,T]$ for almost all $\omega\in\Omega$, and $E\left[\int_{0}^{T}\int_{0}^{T}|\mathbb{D}_{t}^{H}g_{i}(s)|^{2}{\rm d}s{\rm d}t \right]<\infty$. Denote
\begin{eqnarray*}
    Y_{i}(t)=\int_{0}^{t}f_{i}(s){\rm d}s+\int_{0}^{t}g_{i}(s){\rm d}B_{s}^{H}, 0\le t\le T.
\end{eqnarray*}
Then
\begin{eqnarray*}
     Y_{1}(t)Y_{2}(t)=\int_{0}^{t}Y_{1}(s)f_{2}(s){\rm d}s+\int_{0}^{t}Y_{1}(s)g_{2}(s){\rm d}B_{s}^{H}+\int_{0}^{t}Y_{2}(s)f_{1}(s){\rm d}s\quad \quad\quad \quad\quad \\
     +\int_{0}^{t}Y_{2}(s)g_{1}(s){\rm d}B_{s}^{H} +\int_{0}^{t}\mathbb{D}_{s}^{H}Y_{1}(s)g_{2}(s){\rm d}s+\int_{0}^{t}\mathbb{D}_{s}^{H}Y_{2}(s)g_{1}(s){\rm d}s.
\end{eqnarray*}
\end{proposition}

\subsection{Assumptions}
To simplify the presentation, we only discuss the one-dimensional case in this study, assume $(\Omega,\mathcal{F},P)$ is a complete probability space  with natural filtration $\mathcal{F}_{t}$. Consider the following sets: \\

$L^2 (\mathcal{F}_{t};\mathbb{R}):= \bigg\{\varphi :\Omega \rightarrow \mathbb{R}\big|\  \varphi\mbox{ is }  \mathcal{F}_{t}-\mbox{measurable}, E[|\varphi|^2]< \infty  \bigg\};$\\

$L_\mathcal{F} ^2 (0,T;\mathbb{R}):= \bigg\{\varphi :[0,T]\times\Omega\rightarrow \mathbb{R} \big|\  \varphi\mbox{ is progressively measurable process,} $\\ 
$~~~~~~~~~~~~~~~~~~~~~~~~~~~~~~~~~~~~~~~~ E\left[\int_{0}^{T}|\varphi(t)|^2]\,{\rm d}{t}\right]< \infty  \bigg\}$;\\

$\mathcal{C}^{k,l}\left(\left[0,T \right]\times \mathbb{R} \right):= \bigg\{\varphi(t,x) :[0,T]\times\mathbb{R}\rightarrow \mathbb{R} \big|\  \varphi ~\mbox{is} ~k~ \mbox{times differentiable with respect to} $\\
$~~~~~~~~~~~~~~~~~~~~~~~~~~~~t\in [0,T] \mbox{ and} ~l~ \mbox{times coutinuously differentiable with respect to} ~x\in \mathbb{R}$ \bigg\};\\

$\mathcal{C}_{pol}^{k,l}\left(\left[0,T \right]\times \mathbb{R} \right):= \bigg\{\varphi(t,x) \big|\  \varphi \in \mathcal{C}_{pol}^{k,l}\left(\left[0,T \right]\times \mathbb{R} \right), ~\mbox{and all deriratives of } ~\varphi~ \mbox{are of} $\\
$~~~~~~~~~~~~~~~~~~~~~~~~~~~~~~~~~~~~~~~~~~~~~~~~\mbox{polynomial growth}\bigg\}$;\\

$\mathcal{V}_{[0,T]}:=\bigg\{\varphi(t,x) \big|\  \varphi \in \mathcal{C}_{pol}^{1,3}\left(\left[0,T \right]\times \mathbb{R} \right) ~\mbox{with}~\frac{\partial \varphi }{\partial t} \in  \mathcal{C}_{pol}^{0,1}\left(\left[0,T \right]\times \mathbb{R} \right)\bigg\}$;\\

And then, we let $\tilde{\mathcal{V}}_{[0,T+K]}, \tilde{\mathcal{V}}_{[0,T+K]}^{H}$ be the completions of $\mathcal{V}_{[0,T+K]}$ under the following norm, respectively
\begin{eqnarray*}
    \left \| \varphi (\cdot ) \right \|_{\beta } = \left \{E\Big[\int_{0}^{T+K}e^{\beta t}|\varphi_{t}|^{2} {\rm d} t\Big]\right \}^{2},
\end{eqnarray*}
\begin{eqnarray*}
    \left \| \varphi (\cdot ) \right \|_{\beta } = \left \{E\Big[\int_{0}^{T+K}e^{\beta t}t^{2H-1}|\varphi_{t}|^{2} {\rm d} t\Big]\right \}^{2}.
\end{eqnarray*}
where $\beta \ge 0$ is a constant, from Lemma 7 \cite{Mat15} we have $ \tilde{\mathcal{V}}_{[0,T+K]}^{H}\subseteq  \tilde{\mathcal{V}}_{[0,T+K]} \subseteq  \mathcal{V}_{[0,T+K]} \subseteq  \mathbb{L}_{H,[0,T+K]}^{1,2}$.

In addition, we introduce assumptions about $d_i$. 
Let $d_{i}(\cdot), i=1,2,3,4$ represent four $\mathbb{R}^{+}$-valued continuous functions defined on $[0,T]$, and consider the following assumptions:
\begin{enumerate}
               \item[(D1)] There exists a constant $K\ge 0$, such that for all $t\in [0,T]$, $0\le t-d_{1}(t)\le t$, $~0\le t-d_{2}(t)\le t$, 
               $~t\le t+d_{3}(t)\le T+K$, 
               $~t\le t+d_{4}(t)\le T+K$;
               \item[(D2)] There exists a constant $L\ge 0$, such that for all non-negative and integrable $f(\cdot)$,\\
               $\int_{t}^{T}f(s-d_{1}(s))\,{\rm d}{s}\le L\int_{t}^{T+K}f(s){\rm d}{s}$,
               $~\int_{t}^{T}f(s-d_{2}(s))\,{\rm d}{s}\le L\int_{t}^{T+K}f(s){\rm d}{s}$,\\
               $\int_{t}^{T}f(s+d_{3}(s))\,{\rm d}{s}\le L\int_{t}^{T+K}f(s){\rm d}{s}$,
               $~\int_{t}^{T}f(s+d_{4}(s))\,{\rm d}{s}\le L\int_{t}^{T+K}f(s){\rm d}{s}$.
\end{enumerate}

Next, we present assumptions about the generator $f$. Assume that $f(t,\omega,u,v,y,z,\phi,\psi):[0,T]\times\Omega\times L^2 (\mathcal{F}_{s^{\prime}}, \mathbb{R})\times L^2 (\mathcal{F}_{s}, \mathbb{R})\times  \mathbb{R}^{2} \times L^2 (\mathcal{F}_{r^{\prime}}, \mathbb{R})\times L^2 (\mathcal{F}_{r}, \mathbb{R})\to L^2 (\mathcal{F}_{t}, \mathbb{R})$ is a $\mathcal{C}_{pol}^{0,1}$-continuous function, where $0\le s^{\prime}, s \le t \le  r^{\prime}, r\le T+K, t\in[0,T]$, 
satisfy the following two assumptions:
\begin{enumerate}
               \item[(H1)] There exists a constant $C>0$, such that for every $t\in [0,T]$, we have
               \begin{eqnarray*}
                  &&\left|f\left(t,u,v,y,z,\phi,\psi)-f(t,\bar{u},\bar{v},\bar{y},\bar{z},\bar{\phi},\bar{\psi}\right) \right|\\
                  &&\quad \le C\left (|u-\bar{u}|+t^{H-\frac{1}{2}}|v-\bar{v}|+|y-\bar{y}|+t^{H-\frac{1}{2}}|z-\bar{z}|+E\left[|\phi-\bar{\phi}|+t^{H-\frac{1}{2}}|\psi-\bar{\psi}| \Big | \mathcal{F}_t\right] \right ),
               \end{eqnarray*}
               where $u,\bar{u},v,\bar{v}\in L_\mathcal{F}^2 (0,t;\mathbb{R})$; $y,\bar{y},z,\bar{z}\in \mathbb{R}$; $\phi,\bar{\phi},\psi,\bar{\psi}\in L_\mathcal{F}^2 (t,T+K;\mathbb{R})$;
               \item[(H2)] $E\left[ \int_{0}^{T}\left|f(t,0,0,0,0,0,0) \right|^2\rm{d_t} \right]<\infty$, and $ f(t,0,0,0,0,0,0)\in L^{2}_{\mathcal{F}}(0,T+K;\mathbb{R})$.
\end{enumerate}
\section{An existence and uniqueness result for fractional DABSDEs}
If there exists a pair of processes $(Y_t,Z_t)\in \tilde{\mathcal{V}}_{[0,T+K]}\times \tilde{\mathcal{V}}_{[0,T+K]}^{H}$ satisfying the fractional DABSDEs of model (\ref{eq1}), we call $(Y_t,Z_t)$ is a soluition of Eq.(\ref{eq1}).

\begin{theorem}\label{thm1}
Let the assumptions (H1)(H2) be satisfied, $d_{i}(t),i=1,2,3,4,$ satisfy (D1) and (D2), suppose that $\xi_t \in \tilde{\mathcal{V}}_{[T,T+K]}$ and $\eta_t \in  \tilde{\mathcal{V}}_{[T,T+K]}^{H}$, then the fractional DABSDE (\ref{eq1}) has a unique solution $(Y_t,Z_t)_{t\in [0,T+K]}\in \tilde{\mathcal{V}}_{[0,T+K]}\times \tilde{\mathcal{V}}_{[0,T+K]}^{H}$.
\end{theorem}
\begin{proof} 

We rewrite the fractional DABSDE given in Eq.(\ref{eq1}) as
\begin{equation}\label{eq3}
\left\{
\begin{array}{ll}
  \displaystyle Y_{t}=\xi_{T}+\int_{t}^{T}f\left(t,Y_{t-d_{1}(t)},Z_{t-d_{2}(t)},Y_{t},Z_{t},Y_{t+d_{3}(t)},Z_{t+d_{4}(t)}\right){\rm d}{s}-\int_{t}^{T}Z_{t}{\rm d}B_{s}^{H}, \quad 0\leq t\leq T;\\
 \displaystyle Y_{t}=\xi_t, \quad T\leq t\leq T+K;\\
 \displaystyle Z_{t}=\eta_t, \quad T\leq t\leq T+K.
\end{array} \right.
  \end{equation}
  
Then we define the mapping  $\Gamma: \tilde{\mathcal{V}}_{[0,T+K]}\times \tilde{\mathcal{V}}_{[0,T+K]}^{H}\to \tilde{\mathcal{V}}_{[0,T+K]}\times \tilde{\mathcal{V}}_{[0,T+K]}^{H}$ such that $(Y_{\cdot},Z_{\cdot})=\Gamma(y_{\cdot},z_{\cdot})$. For two arbitrary elements $(y_{\cdot},z_{\cdot}),(\bar{y}_{\cdot},\bar{z}_{\cdot})\in \tilde{\mathcal{V}}_{[0,T+K]}\times \tilde{\mathcal{V}}_{[0,T+K]}^{H}$, set $(Y_{\cdot},Z_{\cdot})=\Gamma(y_{\cdot},z_{\cdot})$, $(\bar{Y}_{\cdot},\bar{Z}_{\cdot})=\Gamma(\bar{y}_{\cdot},\bar{z}_{\cdot})$, and we put the differences as follows:
\begin{eqnarray*}  
(\hat{Y}_{\cdot},\hat{Z}_{\cdot}):=(Y_{\cdot}-\bar{Y}_{\cdot},Z_{\cdot}-\bar{Z}_{\cdot}),\quad\quad (\hat{y}_{\cdot},\hat{z}_{\cdot}):=(y_{\cdot}-\bar{y}_{\cdot},z_{\cdot}-\bar{z}_{\cdot}).
\end{eqnarray*}

Now we will prove an estimate.\\
Applying  It\^{o}'s formula for $e^{\beta t}|\hat{Y}_t|^2, t\in[0,T]$, and by using Proposition \ref{proposition2} we have
\begin{eqnarray*}
   &&{\rm d}\left(e^{\beta t}\left|\hat{Y}_t\right|^2\right)=\beta e^{\beta t}\left|\hat{Y}_t\right|^2{\rm d}t+2e^{\beta t}\left|\hat{Y}_t\right|{\rm d}\left|\hat{Y}_t\right|+e^{\beta t}{\rm d}\left|\hat{Y}_t\right|^{2} \\
   &&~~~~~~~~~~~~~~~~~~~~~=-2e^{\beta t}\left|\hat{Y}_t\right| \Big|f\left(t,y_{t-d_{1}(t)},z_{t-d_{2}(t)},y_t,z_t,y_{t+d_{3}(t)},z_{t+d_{4}(t)}\right)\\
   &&~~~~~~~~~~~~~~~~~~~~~~~~~~~~~~~~~~~~~~~~~~~-f\left(t,\bar{y}_{t-d_{1}(t)},\bar{z}_{t-d_{2}(t)},\bar{y}_t,\bar{z}_t,\bar{y}_{t+d_{3}(t)},\bar{z}_{t+d_{4}(t)}\right) \Big|{\rm d}t\\
   &&~~~~~~~~~~~~~~~~~~~~~~~~~~+2e^{\beta t}\left|\hat{Y}_t\right|\left|\hat{Z}_t\right|{\rm d}B_{t}^{H}+2e^{\beta t} \mathbb{D}_{t}^{H}\left|\hat{Y}_t\right|\left|\hat{Z}_t\right|{\rm d}t.
\end{eqnarray*}
Take integral on $[t,T]$ and transpose
\begin{eqnarray*}
    e^{\beta t}\left|\hat{Y}_t\right|^2+\beta\int_{t}^{T} e^{\beta s}\left|\hat{Y}_s\right|^2{\rm d}s+2\int_{t}^{T} e^{\beta s}\left|\hat{Y}_s\right|\left|\hat{Z}_s\right|{\rm d}B_{s}^{H}+2\int_{t}^{T} e^{\beta s} \mathbb{D}_{s}^{H}\left|\hat{Y}_s\right|\left|\hat{Z}_s\right|{\rm d}s\\
    = e^{\beta T}\left|\hat{Y}_T\right|^2+2\int_{t}^{T} e^{\beta s}\left|\hat{Y}_s\right| \Big|f\left(s,y_{s-d_{1}(s)},z_{s-d_{2}(s)},y_s,z_s,y_{s+d_{3}(s)},z_{s+d_{4}(s)}\right)~~~~~\\
    -f\left(t,\bar{y}_{t-d_{1}(t)},\bar{z}_{t-d_{2}(t)},\bar{y}_t,\bar{z}_t,\bar{y}_{t+d_{3}(t)},\bar{z}_{t+d_{4}(t)}\right)\Big|{\rm d}s.
\end{eqnarray*}
Moreover, from equation (9) and Proposition 24 \cite{Mat15}, there exists a suitable constant $M>0$, such that for all $t\in[0,T]$ 
\begin{eqnarray*}
   \frac{t^{2H-1}}{M}Z_t \le \mathbb{D}_{t}^{H}Y_t=\frac{\hat{\sigma}_t}{\sigma_t}Z_t \le  Mt^{2H-1}Z_t,
\end{eqnarray*}
we have 
\begin{eqnarray*}
    e^{\beta t}\left|\hat{Y}_t\right|^2+\beta\int_{t}^{T} e^{\beta s}\left|\hat{Y}_s\right|^2{\rm d}s+2\int_{t}^{T} e^{\beta s}\left|\hat{Y}_s\right|\left|\hat{Z}_s\right|{\rm d}B_{s}^{H}+\frac{2}{M} \int_{t}^{T} e^{\beta s}s^{2H-1} \left|\hat{Z}_s\right|^{2}{\rm d}s\\
    \le e^{\beta T}\left|\hat{Y}_T\right|^2+2\int_{t}^{T} e^{\beta s}\left|\hat{Y}_s\right| \Big|f\left(s,y_{s-d_{1}(s)},z_{s-d_{2}(s)},y_s,z_s,y_{s+d_{3}(s)},z_{s+d_{4}(s)}\right)~~~~~\\
    -f\left(t,\bar{y}_{t-d_{1}(t)},\bar{z}_{t-d_{2}(t)},\bar{y}_t,\bar{z}_t,\bar{y}_{t+d_{3}(t)},\bar{z}_{t+d_{4}(t)}\right)\Big|{\rm d}s.
\end{eqnarray*}
Taking expectation on both sides, applying the fact $2AB\le A^2+B^2$ we have
\begin{eqnarray*}
     &&E\left[ e^{\beta t}\left|\hat{Y}_t\right|^2+\beta\int_{t}^{T} e^{\beta s}\left|\hat{Y}_s\right|^2{\rm d}s
     + \frac{2}{M} \int_{t}^{T} e^{\beta s}s^{2H-1} \left|\hat{Z}_s\right|^{2}{\rm d}s \right]\\
     &&\le  E\bigg[e^{\beta T}\left|\hat{Y}_T\right|^2+2\int_{t}^{T} e^{\beta s}\left|\hat{Y}_s\right| \Big|f\left(s,y_{s-d_{1}(s)},z_{s-d_{2}(s)},y_s,z_s,y_{s+d_{3}(s)},z_{s+d_{4}(s)}\right)\\
    &&~~~~~~~~~~~~~~~~~~~~~~~~~~~~~~~~~~~~~~~-f\left(t,\bar{y}_{t-d_{1}(t)},\bar{z}_{t-d_{2}(t)},\bar{y}_t,\bar{z}_t,\bar{y}_{t+d_{3}(t)},\bar{z}_{t+d_{4}(t)}\right)\Big|{\rm d}s \bigg]\\
    &&\le E\bigg[e^{\beta T}\left|\hat{Y}_T\right|^2 
    + \frac{\beta}{2}\int_{t}^{T}e^{\beta s}\left|\hat{Y}_s\right|^2{\rm d}s
    + \frac{2}{\beta}\int_{t}^{T} e^{\beta s}\Big|f\left(s,y_{s-d_{1}(s)},z_{s-d_{2}(s)},y_s,z_s,y_{s+d_{3}(s)},z_{s+d_{4}(s)}\right)\\
    &&~~~~~~~~~~~~~~~~~~~~~~~~~~~~~~~~~~~~~~~~~~~~~~~~~~~~~~~~~~-f\left(t,\bar{y}_{t-d_{1}(t)},\bar{z}_{t-d_{2}(t)},\bar{y}_t,\bar{z}_t,\bar{y}_{t+d_{3}(t)},\bar{z}_{t+d_{4}(t)}\right)\Big|^2{\rm d}s \bigg].
\end{eqnarray*}
Rearranging the terms again
\begin{eqnarray}\label{inequality}
    &&E\left[ e^{\beta t}\left|\hat{Y}_t\right|^2-e^{\beta T}\left|\hat{Y}_T\right|^2
    +\frac{\beta}{2}\int_{t}^{T}e^{\beta s}\left|\hat{Y}_s\right|^2{\rm d}s
    +\frac{2}{M} \int_{t}^{T} e^{\beta s}s^{2H-1} \left|\hat{Z}_s\right|^{2}{\rm d}s \right]\nonumber\\
     &&\le \frac{2}{\beta}E\bigg[\int_{t}^{T} e^{\beta s}\Big|f\left(s,y_{s-d_{1}(s)},z_{s-d_{2}(s)},y_s,z_s,y_{s+d_{3}(s)},z_{s+d_{4}(s)}\right)\nonumber\\
    &&~~~~~~~~~~~~~~~~~~~~~~~~~~-f\left(t,\bar{y}_{t-d_{1}(t)},\bar{z}_{t-d_{2}(t)},\bar{y}_t,\bar{z}_t,\bar{y}_{t+d_{3}(t)},\bar{z}_{t+d_{4}(t)}\right)\Big|^2{\rm d}s \bigg]
\end{eqnarray}
Next, we will estimate the right part of the above equation, considering the assumption (H1) and Jensen’s inequality
\begin{eqnarray*}
     &&E\bigg[\int_{t}^{T} e^{\beta s}\Big|f\left(s,y_{s-d_{1}(s)},z_{s-d_{2}(s)},y_s,z_s,y_{s+d_{3}(s)},z_{s+d_{4}(s)}\right)\\
    &&~~~~~~~~~~~~~~~~~~-f\left(t,\bar{y}_{t-d_{1}(t)},\bar{z}_{t-d_{2}(t)},\bar{y}_t,\bar{z}_t,\bar{y}_{t+d_{3}(t)},\bar{z}_{t+d_{4}(t)}\right)\Big|^2{\rm d}s \bigg]\\
    &&\le C^{2}E\bigg[\int_{t}^{T} e^{\beta s}\bigg(\left|\hat{y}_{s-d_{1}(s)}\right|+(s-d_{2}(s))^{H-\frac{1}{2}}\left|\hat{z}_{s-d_{2}(s)}\right| +\left|\hat{y}_{s}\right|+s^{H-\frac{1}{2}}\left|\hat{z}_{s}\right| \\
    &&~~~~~~~~~~~+ E\left[\left|\hat{y}_{s+d_{3}(s)}\right|+(s+d_{4}(s))^{H-\frac{1}{2}}\left|\hat{z}_{s+d_{4}(s)}\right|\bigg | \mathcal{F}_t\right]\bigg)^{2} {\rm d}s\bigg]\\
    &&\le 12C^{2}E\bigg[\int_{t}^{T} e^{\beta s}\Big(\left|\hat{y}_{s-d_{1}(s)}\right|^2+(s-d_{2}(s))^{2H-1}\left|\hat{z}_{s-d_{2}(s)}\right|^2 +\left|\hat{y}_{s}\right|^2+s^{2H-1}\left|\hat{z}_{s}\right|^2\\
    &&~~~~~~~~~~~~~~~+\left|\hat{y}_{s+d_{3}(s)}\right|^2+(s+d_{4}(s))^{2H-1}\left|\hat{z}_{s+d_{4}(s)}\right|^2\Big) {\rm d}s\bigg].
\end{eqnarray*}
Substitution inequality (\ref{inequality}), let $t=0$, combine assumptions (D1) and (D2) we have
\begin{eqnarray*}
    &&E\left[ \int_{0}^{T}e^{\beta s}\left( \frac{\beta}{2}\left|\hat{Y}_s\right|^2
    +\frac{2}{M}s^{2H-1} \left|\hat{Z}_s\right|^{2}\right){\rm d}s \right]\\
    &&\le \frac{24C^2}{\beta}E\bigg[\int_{0}^{T} e^{\beta s}\Big(\left|\hat{y}_{s-d_{1}(s)}\right|^2+(s-d_{2}(s))^{2H-1}\left|\hat{z}_{s-d_{2}(s)}\right|^2 +\left|\hat{y}_{s}\right|^2+s^{2H-1}\left|\hat{z}_{s}\right|^2\\
    &&~~~~~~~~~~~~~~~~~~+\left|\hat{y}_{s+d_{3}(s)}\right|^2+(s+d_{4}(s))^{2H-1}\left|\hat{z}_{s+d_{4}(s)}\right|^2\Big) {\rm d}s\bigg]\\
    &&\le \frac{24C^{2}(2L+1)}{\beta}E\bigg[\int_{0}^{T+K} e^{\beta s}\left(\left|\hat{y}_{s}\right|^2+s^{2H-1}\left|\hat{z}_{s}\right|^2  \right){\rm d}s\bigg].
\end{eqnarray*}
Multiply $\frac{M}{2}$ on both sides of the above inequality
\begin{eqnarray*}
     &&E\left[ \int_{0}^{T}e^{\beta s} \left(\frac{M\beta}{4}\left|\hat{Y}_s\right|^2
    +s^{2H-1} \left|\hat{Z}_s\right|^{2}\right){\rm d}s \right]\\
     &&\le \frac{12C^{2}(2L+1)}{\beta}E\bigg[\int_{0}^{T+K} e^{\beta s}\left(\left|\hat{y}_{s}\right|^2+s^{2H-1}\left|\hat{z}_{s}\right|^2  \right){\rm d}s\bigg].
\end{eqnarray*}
Leting $\beta=12C^{2}(2L+1)M+\frac{4}{M}$, we have
\begin{eqnarray*}
    &&E\left[ \int_{0}^{T}e^{\beta s} \left(\left|\hat{Y}_s\right|^2
    +s^{2H-1} \left|\hat{Z}_s\right|^{2}\right){\rm d}s \right]\\
     &&\le \frac{1}{2} E\bigg[\int_{0}^{T+K} e^{\beta s}\left(\left|\hat{y}_{s}\right|^2+s^{2H-1}\left|\hat{z}_{s}\right|^2  \right){\rm d}s\bigg].
\end{eqnarray*}

That is, 
\begin{eqnarray*}
    \left \| (\hat{Y}_{\cdot},\hat{Z}_{\cdot}) \right \| _{\beta}
    \le \frac{1}{\sqrt{2}}\left \| (\hat{y}_{\cdot},\hat{z}_{\cdot}) \right \| _{\beta}.
\end{eqnarray*}

Thus, this mapping $\Gamma$ is a strict contraction on $\tilde{\mathcal{V}}_{[0,T+K]}\times \tilde{\mathcal{V}}_{[0,T+K]}^{H}$; by using the fixed point theorem, the mapping $\Gamma$ has a unique fixed point. That means Eq.(\ref{eq3}) has a unique solution $(Y_t,Z_t)\in \tilde{\mathcal{V}}_{[0,T+K]}\times \tilde{\mathcal{V}}_{[0,T+K]}^{H}$ such that $(Y_t,Z_t)=\Gamma(y_t,z_t)$.
\end{proof}

\section{Comparison Theorem}
In this section, we investigate a comparison theorem of fractional DABSDEs of the one-dimensional kind shown below:
\begin{eqnarray*}
\left\{
\begin{array}{ll}
  \displaystyle Y_{t}=\xi_{T}+\int_{t}^{T}f(s,{Y}_{s-d_{1}(s)},Y_{s},Z_{s},{Y}_{s+d_{3}(s)}) \,{\rm d}{s}- \int_{t}^{T}Z_{s}{\rm d}{B_{s}^{H}}, \quad 0\leq t\leq T;\\
 \displaystyle Y_{t}=\xi_{t}, \quad T\leq t\leq T+K;\\
 \displaystyle Z_{t}=\eta_{t}, \quad T\leq t\leq T+K.
\end{array} \right.
  \end{eqnarray*}
  
Firstly, we introduce the classical case of the comparison theorem of fractional BSDEs; Lemma \ref{lemma2.1} refers to Theorem 12.3 of Hu et al. \cite{Hu12}.
\begin{lemma} \label{lemma2.1}
  Define $\eta_t=\eta_0+\int_{0}^{t}b_s{\rm d}s+\int_{0}^{t}\sigma_s{\rm d}B_{t}^{H}$, where $b_s$ and $\sigma_s$ are bounded deterministic functions and $\sigma_s>0$. For $j=1,2$, assume $\xi_{T}^{j}$ are  continuously differentiable and is of
  polynomial growth, $f_{j}\left(t,x,y,z\right)$ and $\frac{\partial }{\partial y} f_{j}\left(t,x,y,z\right)$  are uniformly Lipschitz continuous with respect to $y$ and $z$. Let$(y^{(1)}, z^{(1)}),(y^{(2)}, z^{(2)})$ be the solutions of the following classical type of fractional BSDE: 
  \begin{eqnarray*}
  \left\{
\begin{array}{ll}
  \displaystyle{\rm d}y_{t}^{(j)}= - f_{j}\left(t,\eta_t,y_{t}^{(j)},z_{t}^{(j)}\right){\rm d}t+z_{t}^{(j)}{\rm d}B_{t}^{H},\\
  \displaystyle y_{T}^{(j)}=\xi_{T}^{(j)}.
  \end{array} \right.
  \end{eqnarray*}
  If $\xi_{T}^{(1)}\le \xi_{T}^{(2)}, f_{1}(t,x,y,z)\le f_{2}(t,x,y,z), t\in [0,T]$, then 
 \begin{eqnarray*}
  y_{t}^{(1)}\le y_{t}^{(2)}, ~~~~~a.e., a.s.
  \end{eqnarray*}  
\end{lemma}

Next, let $(Y_{\cdot}^{(1)}, Z_{\cdot}^{(1)}),(Y_{\cdot}^{(2)}, Z_{\cdot}^{(2)})$ be  the solutions of the two one-dimensional fractional DABSDEs shown below, respectively, 
\begin{equation}\label{eq2.1}
\left\{
\begin{array}{ll}
  \displaystyle Y_{t}^{(j)}=\xi_{T}^{(j)}+\int_{t}^{T}f_{j}\left(s,Y^{(j)}_{s-d_{1}(s)},Y^{(j)}_{s},Z^{(j)}_{s},Y^{(j)}_{s+d_{3}(s)}\right) {\rm d}{s}- \int_{t}^{T}Z^{(j)}_{s}{\rm d}{B_{s}^{H}}, \quad 0\leq t\leq T;\\
 \displaystyle Y_{t}^{(j)}=\xi_{t}^{(j)}, \quad T\leq t\leq T+K;\\
 \displaystyle Z_{t}^{(j)}=\eta_{t}^{(j)}, \quad T\leq t\leq T+K,
\end{array} \right.
  \end{equation}
 where $j =1,2$. The end outcome is as follows.
 
\begin{theorem}\label{thm2}
 Suppose $f_{j}(t,\cdot), j =1,2$ satisfy the assumptions (H1) and (H2), $\xi_{t}^{(j)}\in \tilde{\mathcal{V}}_{[T,T+K]}$, and $d_{i}(t), i=1,2,3,4$ satisfy (D1) and (D2). Moreover, assume that
 \begin{enumerate}
     \item [(i)] $f_{2}(t,u,y,z,\phi)$ is increasing with respect to $u$ and $\phi$;
     \item[(ii)] $\xi_{t}^{(1)}\le \xi_{t}^{(2)}$;
     \item[(iii)] $f_{1}\left(t,y_{t-d_{1}(t)},y_{t},z_{t},y_{t+d_{3}(t)}\right)\le f_{2}\left(t,y_{t-d_{1}(t)},y_{t},z_{t},y_{t+d_{3}(t)}\right)$,\\
     $y_{t-d_{1}(t)}\in L_\mathcal{F}^2 (0,t), y_{t+d_{3}(t)}\in L_\mathcal{F}^2 (t,T+K)$.
 \end{enumerate}
Then we have $Y_{t}^{(1)}\le Y_{t}^{(2)}$ almost surely.
\end{theorem}
\begin{proof}
Consider the following BSDE
\begin{equation}\label{eq2.2}
\left\{
\begin{array}{ll}
  \displaystyle Y_{t}^{(1)}=\xi_{T}^{(1)}+\int_{t}^{T}f_{1}\left(s,Y^{(1)}_{s-d_{1}(s)},Y^{(1)}_{s},Z^{(1)}_{s},Y^{(1)}_{s+d_{3}(s)}\right)  \,{\rm d}{s}
  - \int_{t}^{T}Z^{(1)}_{s}{\rm d}{B_{s}^{H}}, \quad 0\leq t\leq T;\\
 \displaystyle Y_{t}^{(1)}=\xi_{t}^{(1)}, \quad T\leq t\leq T+K.
\end{array} \right.
  \end{equation}
As we know, $\left(Y_{\cdot}^{(1)},Z_{\cdot}^{(1)}\right)$ is the solution of the one-dimensional fractional DABSDEs given in Eq.(\ref{eq2.1}).
Next, we consider the following BSDE:
\begin{equation}\label{eq2.3}
\left\{
\begin{array}{ll}
  \displaystyle Y_{t}^{(3)}=\xi_{T}^{(2)}+\int_{t}^{T}f_{2}(s,Y^{(1)}_{s-d_{1}(s)},Y^{(3)}_{s},Z^{(3)}_{s},Y^{(1)}_{s+d_{3}(s)})  \,{\rm d}{s}- \int_{t}^{T}Z^{(3)}_{s}{\rm d}{B_{s}^{H}}, \quad 0\leq t\leq T;\\
 \displaystyle Y_{t}^{(3)}=\xi_{t}^{(2)}, \quad T\leq t\leq T+K.
\end{array} \right.
  \end{equation}
  We see the above equation has a unique solution $\left(Y_{t}^{(3)},Z_{t}^{(3)}\right)\in \tilde{\mathcal{V}}_{[0,T+K]}\times \tilde{\mathcal{V}}_{[0,T]}^{H}$.
Since $\xi_{t}^{(1)}\le \xi_{t}^{(2)}$, $f_{1}\left(s,Y^{(1)}_{s-d_{1}(s)},y,z,Y^{(1)}_{s+d_{3}(s)}\right)\le f_{2}\left(s,Y^{(1)}_{s-d_{1}(s)},y,z,Y^{(1)}_{s+d_{3}(s)}\right)$, considering Eq.(\ref{eq2.2}) and Eq.(\ref{eq2.3}), by Lemma \ref{lemma2.1}, we have 
\begin{eqnarray*}
    Y_{t}^{(1)}\le Y_{t}^{(3)}~~a.e., a.s.
\end{eqnarray*}

Set
\begin{equation}\label{eq2.4}
\left\{
\begin{array}{ll}
  \displaystyle Y_{t}^{(4)}=\xi_{T}^{(2)}+\int_{t}^{T}f_{2}\left(s,Y^{(3)}_{s-d_{1}(s)},Y^{(4)}_{s},Z^{(4)}_{s},Y^{(3)}_{s+d_{3}(s)}\right) {\rm d}{s}- \int_{t}^{T}Z^{(4)}_{s}{\rm d}{B_{s}^{H}}, \quad 0\leq t\leq T;\\
 \displaystyle Y_{t}^{(4)}=\xi_{t}^{(2)}, \quad T\leq t\leq T+K.
\end{array} \right.
  \end{equation}
  
Consider Eq.(\ref{eq2.3}) and Eq.(\ref{eq2.4}); $f_{2}\left(t,u,y,z,\phi\right)$ is increasing in $u$ and $\phi$, and $Y_{t}^{(1)}\le Y_{t}^{(3)}$, which imply $f_{2}\left(s,Y^{(1)}_{s-d_{1}(s)},y,z,Y^{(1)}_{s+d_{3}(s)}\right)\le f_{2}\left(s,Y^{(3)}_{s-d_{1}(s)},y,z,Y^{(3)}_{s+d_{3}(s)}\right)$. Similar to the above, we have
\begin{eqnarray*}
     Y_{t}^{(3)}\le Y_{t}^{(4)}~~a.e., a.s.
\end{eqnarray*}

For $n=5,6,\dots $, we consider the following BSDE:
\begin{eqnarray*}
\left\{
\begin{array}{ll}
  \displaystyle Y_{t}^{(n)}=\xi_{T}^{(2)}+\int_{t}^{T}f_{2}\left(s,Y^{(n-1)}_{s-d_{1}(s)},Y^{(n)}_{s},Z^{(n)}_{s},Y^{(n-1)}_{s+d_{3}(s)}\right)  {\rm d}{s}- \int_{t}^{T}Z^{(n)}_{s}{\rm d}{B_s}, \quad 0\leq t\leq T;\\
 \displaystyle Y_{t}^{(n)}=\xi_{t}^{(2)}, \quad T\leq t\leq T+K.
\end{array} \right.
  \end{eqnarray*}
  
Similarly, we get 
\begin{eqnarray*}
     Y_{t}^{(4)}\le Y_{t}^{(5)}\le \dots \le Y_{t}^{(n)}\le \dots ,~~a.s.
\end{eqnarray*}

Next, we will show that for $n\ge 4$, $\left \{ Y_{t}^{(n)}, Z_{t}^{(n)}  \right \} $ is a Cauchy sequences. Denote $\hat{Y}_{t}^{(n)}:=Y_{t}^{(n)}-Y_{t}^{(n-1)}$, $\hat{Z}_{t}^{(n)}:=Z_{t}^{(n)}-Z_{t}^{(n-1)}$, $n\ge 4$, then from estimate (\ref{inequality}), we get
\begin{eqnarray*}
    &&E\left[ e^{\beta t}\left|\hat{Y}_{t}^{(n)}\right|^2-e^{\beta T}\left|\hat{Y}_{T}^{(n)}\right|^2
    +\frac{\beta}{2}\int_{t}^{T}e^{\beta s}\left|\hat{Y}_{s}^{(n)}\right|^2{\rm d}s
    +\frac{2}{M} \int_{t}^{T} e^{\beta s}s^{2H-1} \left|\hat{Z}_{s}^{(n)}\right|^{2}{\rm d}s \right]\nonumber\\
     &&\le \frac{2}{\beta}E\bigg[\int_{t}^{T} e^{\beta s}\Big|f_{2}\left(s,Y_{s-d_{1}(s)}^{(n-1)},Y_{s}^{(n)},Z_{s}^{(n)},Y_{s+d_{3}(s)}^{(n-1)}\right)\nonumber\\
    &&~~~~~~~~~~~~~~~~~~~~~~~~~~-f_{2}\left(s,Y_{s-d_{1}(s)}^{(n-2)},Y_{s}^{(n-1)},Z_{s}^{(n-1)},Y_{s+d_{3}(s)}^{(n-2)}\right)\Big|^{2}{\rm d}s \bigg]
\end{eqnarray*}

When we let $t=0$, apply Jensen's inequality, assumptions (H1), (D1) and (D2), and the fact that $(a+b+c+d)^{2}\le 4(a^{2}+b^{2}+c^{2}+d^{2})$, one has
\begin{eqnarray*}
   && E\left[\int_{0}^{T}e^{\beta s}\left(\frac{\beta}{2}\left |\hat{Y}_{s}^{(n)}\right |^2+\frac{2}{M}s^{2H-1}\left|\hat{Z}_{s}^{(n)}\right|^2 \right){\rm d}s \right] \\
   && \le \frac{2C^{2}}{\beta}E\left[\int_{0}^{T}e^{\beta s} \bigg(\left|\hat{Y}_{s}^{(n-1)}\right|+\left|\hat{Y}_{s}^{(n)}\right|+s^{H-\frac{1}{2}}\left|\hat{Z}_{s}^{(n)}\right|+E\left[\left|\hat{Y}_{t}^{(n-1)}\right| \bigg| \mathcal{F}_t\right]\bigg)^{2}{\rm d}s\right]\\
   && \le \frac{8C^{2}(2L+1)}{\beta}E\left[\int_{0}^{T}e^{\beta s} \Big(\left|\hat{Y}_{s}^{(n-1)}\right|^{2}+\left|\hat{Y}_{s}^{(n)}\right|^{2}+s^{2H-1}\left|\hat{Z}_{s}^{(n)}\right|^{2}+\left|\hat{Y}_{s}^{(n-1)}\right|^{2} \Big){\rm d}s\right].
\end{eqnarray*}
Multiply by $\frac{M}{2}$
\begin{eqnarray*}
     && E\left[\int_{0}^{T}e^{\beta s}\left(\frac{M\beta}{4}\left |\hat{Y}_{s}^{(n)}\right |^2+s^{2H-1}\left|\hat{Z}_{s}^{(n)}\right|^2 \right){\rm d}s \right] \\
    && \le \frac{4C^{2}(2L+1)M}{\beta}E\left[\int_{0}^{T}e^{\beta s} \Big(\left|\hat{Y}_{s}^{(n-1)}\right|^{2}+\left|\hat{Y}_{s}^{(n)}\right|^{2}+s^{2H-1}\left|\hat{Z}_{s}^{(n)}\right|^{2}+\left|\hat{Y}_{s}^{(n-1)}\right|^{2} \Big){\rm d}s\right].
\end{eqnarray*}
Let $\beta=16MC^{2}(2L+1)+\frac{4}{M}, M>2$, then we obtain
\begin{eqnarray*}
    &&E\left[\int_{0}^{T}e^{\beta s}\left(\left|\hat{Y}_{s}^{(n)}\right|^2+s^{2H-1}\left|\hat{Z}_{s}^{(n)}\right|^2 \right){\rm d}s \right]\\ 
    &&\le \frac{1}{4}E\left[ \int_{0}^{T}e^{\beta s} \Big(\left|\hat{Y}_{s}^{(n)}\right|^{2}+s^{2H-1}\left|\hat{Z}_{s}^{(n)}\right|^{2}\Big){\rm d}s\right]  +\frac{1}{2}E\left [ \int_{0}^{T}e^{\beta s}\left|\hat{Y}_{s}^{(n-1)}\right|^{2}{\rm d}s \right]. 
\end{eqnarray*}

Transpose,
\begin{eqnarray*}
   &&E\left[\int_{0}^{T}e^{\beta s}\left(\left|\hat{Y}_{s}^{(n)}\right|^2+s^{2H-1}\left|\hat{Z}_{s}^{(n)}\right|^2 \right){\rm d}s \right]\\ 
   &&\le \frac{2}{3}E\left [ \int_{0}^{T}e^{\beta s}\left|\hat{Y}_{s}^{(n-1)}\right|^{2}{\rm d}s \right]\\
   &&\le \frac{2}{3}E\left [ \int_{0}^{T}e^{\beta s}\left(\left|\hat{Y}_{s}^{(n-1)}\right|^{2}+s^{2H-1}\left|\hat{Z}_{s}^{(n-1)}\right|^2\right){\rm d}s \right].
\end{eqnarray*}

Therefore,
\begin{eqnarray*}
     E\left[\int_{0}^{T}e^{\beta s}\left(\left|\hat{Y}_{s}^{(n)}\right|^2+s^{2H-1}\left|\hat{Z}_{s}^{(n)}\right|^2 \right){\rm d}s \right] 
     \le  \left(\frac{2}{3} \right)^{n-4}E\left[\int_{0}^{T}e^{\beta s}\left(\left|\hat{Y}_{s}^{(4)}\right|^2+s^{2H-1}\left|\hat{Z}_{s}^{(4)}\right|^2 \right){\rm d}s \right].
\end{eqnarray*}

This means $\left(\hat{Y}_{t}^{(n)},\hat{Z}_{t}^{(n)}\right)_{n\ge 4}$ is Cauchy sequence in $\tilde{\mathcal{V}}_{[0,T+K]}\times \tilde{\mathcal{V}}_{[0,T]}^{H}$. Let $(Y_{\cdot},Z_{\cdot})$ be the limit of $\left(\hat{Y}_{t}^{(n)},\hat{Z}_{t}^{(n)}\right)$ for all $0\le t \le T$, when $n\to \infty$
\begin{eqnarray*}
    E\left[ \int_{t}^{T}e^{\beta s}\left|f_{2}\left(s,Y_{s-d_{1}(s)}^{(n-1)},Y_{s}^{(n)},Z_{s}^{(n)},Y_{s+d_{3}(s)}^{(n-1)}\right)-f_{2}\left(s,Y_{s-d_{1}(s)},Y_{s},Z_{s},Y_{s+d_{3}(s)}\right)\right|^2\,{\rm d}{s}\right]\\
    \le 4C^{2} E\left[ \int_{t}^{T}e^{\beta s}\left( \left|Y_{s}^{(n)}-Y_{s}\right|^{2}+s^{2H-1}\left|Z_{s}^{(n)}-Z_{s}\right|^{2}+2L\left|Y_{s}^{(n-1)}-Y_{s}\right|^{2} \right){\rm d}{s}\right] \longrightarrow  0.
\end{eqnarray*}

Consequently, $(Y_{t},Z_{t})$ is a solution of the following fractional DABSDE:
\begin{eqnarray*}
\left\{
\begin{array}{ll}
  \displaystyle Y_{t}=\xi_{T}^{(2)}+\int_{t}^{T}f_{2}(s,Y_{s-d_{1}(s)},Y_{s},Z_{s},Y_{s+d_{3}(s)}) {\rm d}{s}- \int_{t}^{T}Z_{s}{\rm d}{B_{s}^{H}}, \quad 0\leq t\leq T;\\
 \displaystyle Y_{t}=\xi_{t}^{(2)}, \quad T\leq t\leq T+K.
\end{array} \right.
  \end{eqnarray*}
  
Then, by Theorem \ref{thm1} on the uniqueness of the solution, we have
\begin{eqnarray*}
    Y_{t}^{(2)}=Y_{t},\quad  a.s.
\end{eqnarray*}

Since 
\begin{eqnarray*}
    Y_{t}^{(1)}\le Y_{t}^{(3)}\le Y_{t}^{(4)}\le Y_{t},
\end{eqnarray*}
then we obtain the desired result $Y_{t}^{(1)}\le Y_{t}^{(2)}, a.s.$
\end{proof}


\begin{thebibliography}{999}
\bibitem{Par90}
Pardoux, E.; Peng, S. Adapted solution of a backward stochastic differential equation. {\em Systems Control Lett.} {\bf 1990}, {\em 14}, 55--61. 

\bibitem{Ba02}
Bahlali, K; Essaky, E.H.; Oukine, Y. Reflected backward stochastic differential equation with jumps and locally Lipschitz coefficient. {\em Random Oper. Stoch. Equ.} {\bf 2002}, {\em 10}, 481–486. 
\bibitem{Abd22}
Abdelhadi, K.; Eddahbi, M.; Khelfallah, N.; Almualim, A. Backward Stochastic Differential Equations Driven by a Jump Markov Process with Continuous and Non-Necessary Continuous Generators. {\em RFractal Fract. } {\bf 2022}, {\em 6}, 331. 

\bibitem{Zhang22}
Zhang, P.; Ibrahim, A.I.N.; Mohamed, N.A. Backward Stochastic Differential Equations (BSDEs) Using Infinite-Dimensional Martingales with Subdifferential Operator. {\em Axioms.} {\bf 2022}, {\em 11}, 536. 

\bibitem{Ma02}
Ma, J.; Protter, P.; Martín, J.S.; Torres, S. Numberical Method for Backward Stochastic Differential Equations. {\em Ann. Appl. Probab.} {\bf 2002}, {\em 12}, 302-316. 


\bibitem{Go05}
Gobet, E.; Lemor, J.P.; Warin, X. A regression-based Monte Carlo method to solve backward stochastic differential equations. {\em Ann. Appl. Probab.} {\bf 2005}, {\em 15}, 2172–2202. 

\bibitem{Zhao14}
Zhao, W.; Zhang, W.; Ju, L. A Numerical Method and its Error Estimates for the Decoupled Forward-Backward Stochastic Differential Equations.
{\em Commun. Comput. Phys.} {\bf 2014}, {\em 15}, 618–646. 


\bibitem{Ren06}
Ren, Y.; Xia, N. Generalized Reflected BSDE and an Obstacle Problem for PDEs with a Nonlinear Neumann Boundary Condition. {\em Stoch. Anal. Appl.} {\bf 2006}, {\em 24}, 1013–1033. 

\bibitem{Par14}
Pardoux, E.; R{\u{a}}{\c{s}}canu, A. Backward Stochastic Differential Equations. {\em In Stochastic Modelling and Applied Probability; Springer International Publishing} {\bf 2014}, {\em }, 353–515. 


\bibitem{Kar97}
Karoui, N.E.; Peng, S.; Quenez, M.C. Backward stochastic differential equations in finance. {\em Math. Finance} {\bf 1997}, {\em 7}, 1-71.

\bibitem{Peng99}
Peng, S.; Wu, Z. Fully Coupled Forward-Backward Stochastic Differential Equations and Applications to Optimal Control. {\em SIAM J. Control Optim.} {\bf 1999}, {\em 37}, 825–843. 

\bibitem{Li09}
Li, J.; Peng, S. Stochastic optimization theory of backward stochastic differential equations with jumps and viscosity solutions of Hamilton–Jacobi–Bellman equations. {\em Nonlinear Anal.} {\bf 2009}, {\em 70}, 1776–1796. 

\bibitem{Peng09}
Peng, S.; Yang, Z. Anticipated backward stochastic differential equations. {\em Ann. Appl. Probab.} {\bf 2009}, {\em 37},  

\bibitem{Feng16}
Feng, X. Anticipated Backward Stochastic Differential Equation with Reflection. {\em Comm. Statist. Simulation Comput.} {\bf 2016}, {\em 45}, 1676–1688. 

\bibitem{Zhang23}
Zhang P.; Mohamed N. A.; Ibrahim A. I. N. Mean-Field and Anticipated BSDEs with Time-Delayed Generator. {\em Mathematics.} {\bf 2023}, {\em 11}, 888. 

\bibitem{Wang22}
Wang, T.; Cui, S. Anticipated Backward Doubly Stochastic Differential Equations with Non-Lipschitz Coefficients. {\em Mathematics} {\bf 2022}, {\em 10}, 396. 

\bibitem{Del10}
Delong, Ł.; Imkeller, P. Backward stochastic differential equations with time delayed generators—results and counterexamples. {\em Ann. Appl. Probab.} {\bf 2010}, {\em 20},  1512-1536.

\bibitem{He20}
He, P.; Ren, Y.; Zhang, D. A Study on a New Class of Backward Stochastic Differential Equation. {\em Math. Probl.  Eng.} {\bf 2020}, {\em 2020}, 1–9.

\bibitem{Zhuang17}
Zhuang, Y. Non-zero sum differential games of anticipated forward-backward stochastic differential delayed equations under partial information and application. {\em Adv. Difference Equ.} {\bf 2017}, {\em 2017}, 1-21. 
\bibitem{Kol40}
Kolmogorov A. N. Wienersche spiralen und einige andere interessante kurven in hilbertscen raum, cr (doklady). {\em  Acad. Sci. URSS (NS)} {\bf 1940}, {\em 26}, 115-118. 

\bibitem{Hu09}
Hu, Y.; Peng S. Backward stochastic differential equation driven by fractional Brownian motion. {\em SIAM J. Control Optim.} {\bf 2009}, {\em 48}, 1675-1700. 

\bibitem{Bor13}
Borkowska, K. Generalized BSDEs driven by fractional Brownian motion. {\em Statist. Probab. Lett.} {\bf 2013}, {\em 83}, 805–
811. 

\bibitem{Dou19}
Douissi, S.; Wen, J.; Shi, Y. Mean-field anticipated BSDEs driven by fractional Brownian motion and related stochastic control problem. {\em Appl. Math. Comput.} {\bf  2019}, {\em 355}, 282–298. 

\bibitem{Wen17}
Wen, J.; Shi, Y. F. Anticipative backward stochastic differential equations driven by fractional Brownian motion. {\em  Statist. Probab. Lett.} {\bf 2017}, {\em 122}, 118–127. 

\bibitem{Wen22}
Wen, J.Wen J. Fractional backward stochastic differential equations with delayed generator. {\em  arXiv.} {\bf 2022}, {\em 2211}, 16826. 
\bibitem{Dec99}
Decreusefond, L.; Üstünel, A. S. Stochastic analysis of the fractional Brownian motion. {\em Potent. Anal. } {\bf 1999}, {\em 10}, 177–214. 

\bibitem{Dun00}
Duncan, T. E.; Hu, Y.; Pasik-Duncan, B. Integral transformations and anticipative calculus for fractional Brownian motions. {\em SIAM J. Control Optim.} {\bf 2000}, {\em 38}, 582-612. 

\bibitem{Hu05}
Hu, Y. Integral transformations and anticipative calculus for fractional Brownian motions. {\em Mem. Amer. Math. Soc.} {\bf 2005}, {\em 175}, pp. 
\bibitem{Mat15}
Maticiuc, L.; Nie, T. Fractional backward stochastic differential equations and fractional backward variational inequalities.  {\em  J. Theoret. Probab.} {\bf 2015}, {\em 28}, 337–395. 

\bibitem{Hu12}
Hu, Y.; Ocone, D.; Song, J. Some results on backward stochastic differential equations driven by fractional Brownian motions. {\em  Stoch. Anal. Appl. Finance} {\bf 2012}, {\em 13}, 225–242. 


\end{thebibliography}
\end{document}